\RequirePackage[l2tabu, orthodox]{nag}
\documentclass[a4paper,11pt]{amsart}
\usepackage[T1]{fontenc}
\usepackage[utf8]{inputenc}
\usepackage{amsmath}
\usepackage{amssymb}
\usepackage{amsthm}
\usepackage{ellipsis, microtype}
\usepackage{enumerate}
\usepackage[margin=0.75in]{geometry}
\usepackage[unicode,pdftex]{hyperref}
\hypersetup{draft=false,colorlinks=true,pdfpagemode=UseOutlines,plainpages=false,
            breaklinks=true,bookmarksnumbered=true,pdfborder=0 0 0}
\usepackage{tikz}

\newtheorem{mainthm}{Theorem}

\newtheorem{theorem}{Theorem}[section]
\newtheorem{definition}[theorem]{Definition}
\newtheorem{lemma}[theorem]{Lemma}
\newtheorem{proposition}[theorem]{Proposition}
\newtheorem{corollary}[theorem]{Corollary}
\newtheorem{remark}[theorem]{Remark}
\newtheorem{examplecore}[theorem]{Example}

\newenvironment{example}{\begin{examplecore}}{\hspace*{\fill}
$\square$\par\vspace{.1cm}\end{examplecore}}

\newcommand{\op}{\operatorname}

\newcommand{\Ao}{\ensuremath{\mathbb{A}^1}}
\newcommand{\Ga}{\ensuremath{\mathbb{G}_{\op{a}}}}
\newcommand{\Gm}{\ensuremath{\mathbb{G}_{\op{m}}}}
\newcommand{\ZZ}{\ensuremath{\mathbb{Z}}}
\newcommand{\St}{\ensuremath{\op{St}}}
\newcommand{\Sing}{\ensuremath{\op{Sing}_\bullet^{\Ao}}}

\newcommand{\exinf}{\ensuremath{\op{Ex}^\infty_{\Ao}}}
\newcommand{\epito}{\ensuremath{\twoheadrightarrow}}
\newcommand{\monoto}{\ensuremath{\hookrightarrow}}
\newcommand{\isoto}{\ensuremath{\stackrel{\sim}{\longrightarrow}}}
\newcommand{\Ker}{\ensuremath{\op{ker}}}

\AtBeginDocument{%

\usepackage{prettyref}
\newrefformat{thm}{\hyperref[{#1}]{Theorem~\ref*{#1}}}
\newrefformat{def}{\hyperref[{#1}]{Definition~\ref*{#1}}}
\newrefformat{lem}{\hyperref[{#1}]{Lemma~\ref*{#1}}}
\newrefformat{prop}{\hyperref[{#1}]{Proposition~\ref*{#1}}}
\newrefformat{cor}{\hyperref[{#1}]{Corollary~\ref*{#1}}}
\newrefformat{rem}{\hyperref[{#1}]{Remark~\ref*{#1}}}
\newrefformat{ex}{\hyperref[{#1}]{Example~\ref*{#1}}}

}

\begin{document}

\title{On $\Ao$-fundamental groups of isotropic reductive groups}

\author{Konrad Voelkel and Matthias Wendt}

\date{October 2015}

\address{Konrad Voelkel, Mathematisches Institut, Albert-Ludwigs-Uni\-ver\-si\-t\"at Freiburg, Eckerstra\ss{}e 1, 79104, Freiburg im Breisgau, Germany}
\email{konrad.voelkel@math.uni-freiburg.de}
\thanks{Konrad Voelkel acknowledges support by the DFG-Forschergruppe 570 ``Algebraic cycles and L-functions''.}
\address{Matthias Wendt, Fakult\"at f\"ur Mathematik, Universit\"at Duisburg-Essen, Thea-Leymann-Strasse 9, 45127 Essen}
\email{matthias.wendt@uni-due.de}


\begin{abstract}
For an isotropic reductive group $G$ satisfying a suitable rank condition over an infinite field $k$, we show that the sections of the $\Ao$-fundamental group sheaf of $G$ over an extension field $L/k$ can be identified with the second group homology of $G(L)$. For a split group $G$, we provide explicit loops representing all elements in the $\Ao$-fundamental group. Using $\Ao$-homotopy theory, we deduce a Steinberg relation for these explicit loops.
\end{abstract}


\maketitle

\section{Introduction}
\label{sec:intro}

The goal of the present note is to describe the $\Ao$-fundamental group sheaves for isotropic reductive groups, improving the computations of 
\cite[Proposition 5.2]{chev-rep}. Moreover, for split groups, we obtain more precise information on the $\Ao$-fundamental groups by providing explicit loops representing elements in the $\Ao$-fundamental groups. The precise statement of our result is the following, cf. \prettyref{lem:bgiso} and \prettyref{prop:pi1h2}:

\begin{mainthm}
Let $k$ be an infinite field and let $G$ be an isotropic reductive group over $k$, assuming that all components of the relative root system of $G$ have at least rank $2$. Then there is an isomorphism
\[
\op{H}_2(G(k),\ZZ) \cong \pi_1(G(k[\Delta^\bullet])) \cong \pi_1^{\Ao}(G)(k).
\]
In the case of split $G$, the isomorphism can be described by an explicit map
\[
K_2^{\op{M(W)}}(k) \isoto \pi_1(G(k[\Delta^\bullet])).
\]
\end{mainthm}

To prove the result, we use homotopy invariance of group homology, cf. \cite{hoinv}, and a definition of Steinberg groups based on work of Petrov and Stavrova to identify $\op{H}_2(G(k),\ZZ)$ with $\pi_1(G(k[\Delta^\bullet]))$.
The results of \cite{gbun1} and \cite{gbun2} on affine excision and descent for isotropic groups relate the latter to $\Ao$-homotopy theory. A slightly different approach is described in \prettyref{rem:earlier}. The Steinberg relation for explicit loops in $\op{H}_2(G(k),\ZZ)$ follows from arguments of Hu and Kriz. 

Using Morel's theory of strictly $\Ao$-invariant sheaves \cite{morel:a1algtop}, we also get the following: 
\begin{corollary}
Let $k$ be an infinite perfect field and let $G$ be as above. Then the assignment $L/k\mapsto \op{H}_2(G(L),\ZZ)$ extends to a strictly $\Ao$-invariant sheaf of abelian groups. 
\end{corollary}

Another implication of the above theorem is that Rehmann's computation of $\op{H}_2(\op{SL}_n(D),\ZZ)$, cf. \cite{rehmann}, can be seen as a description of $\pi_1^{\Ao}(\op{SL}_n(D))$, for $n\geq 3$. The corollary implies the existence of well-behaved residue maps on $\op{H}_2(\op{SL}_n(D),\ZZ)$ which seem to be new.

\emph{Acknowledgments:} The results presented here originate from the first named author's Diplomarbeit \cite{voelkel:diplom} at the Albert-Ludwigs-Universit\"at Freiburg.

\section{Preliminaries} 
\label{sec:prelims}

In this article, we always assume $k$ to be an infinite field. We consider reductive groups $G$ over $k$, and we assume that they are \emph{isotropic}, as in \cite{petrov:stavrova:elementary}, so that all irreducible components of the relative root system of such $G$ are of rank at least $2$. This implies that the results of \cite{petrov:stavrova:elementary} and \cite{gbun2} are applicable. 

For a commutative unital $k$-algebra $R$, the (abstract) group of $R$-points of the group scheme $G$ is denoted by $G(R)$. The elementary subgroup $\op{E}(R)\subset G(R)$ is defined, as in \cite[\S1]{petrov:stavrova:elementary}, to be the subgroup of $G(R)$ generated by $R$-points of unipotent radicals of opposite parabolics $P^+,P^-$ of $G$. By \cite[Theorem 1]{petrov:stavrova:elementary}, $\op{E}(R)$ is normal in $G(R)$, and by \cite[Theorem 1]{luzgarev:stavrova}, the group $\op{E}(R)$ is perfect. Moreover, by \cite[Theorem 1.3]{stavrova}, $K_1^G(R):=G(R)/\op{E}(R)$ is invariant under polynomial extensions. 

\begin{definition}
Let $G$ be an isotropic reductive group over a commutative ring $R$.
We define the \emph{Steinberg group} $\St^G(R)$ to be the abstract group generated by elements $\widetilde{X_A}(u)$,
$u\in V_A(R)$ subject to the commutator formulas from
\cite[Lemma 9, 10]{petrov:stavrova:elementary}. We define the group $K_2^G(R):=\Ker\left(\St^G(R)\to \op{E}^G(R)\right)$. 
\end{definition}

\begin{remark}
It is known that $K_2^G(k[\Delta^n]) \monoto \St^G(k[\Delta^n]) \epito E^G(k[\Delta^n])$ is a universal central extension for $G$ split of type $A_l, l\geq 3$ (van der Kallen), $C_l, l \geq 3$ (Lavrenov) and $E_l$ (Sinchuk). It is not even a central extension for split rank $2$ groups.
\end{remark}

Using the standard cosimplicial object given by polynomial rings, one can associate a simplicial group to the reductive group $G$ and a unital commutative $k$-algebra $A$, cf. \cite{jardine:homotopy}. This is denoted by $G(A[\Delta^\bullet])$ or (more commonly in the $\Ao$-homotopy literature) by $\Sing(G)(A)$. The $\Ao$-homotopy groups of an isotropic reductive group can be computed from the singular resolution, cf. \cite[Corollary 4.3.3]{gbun2}.

\begin{lemma}
\label{lem:bgiso}
Let $k$ be an infinite field and let $G$ be an isotropic reductive group over $k$.

Then $\Sing(G)$ has affine Nisnevich excision in the sense of \cite[Definition 3.2.1]{gbun1} and there are isomorphisms
\[
\pi_i(\Sing(G)(A)) \isoto \pi_i^{\Ao}(G)(A)
\]
for any essentially smooth $k$-algebra $A$ and any $i\geq 0$.
\end{lemma}

\begin{remark}\label{rem:earlier}
Alternatively,
one can prove affine Nisnevich excision exactly as in \cite[Theorem 4.10]{chev-rep}, using homotopy invariance for unstable $K_1^G$ of isotropic groups from \cite[Theorem 1.3]{stavrova}. The above result then follows from the general representability result \cite[Theorem 3.3.5]{gbun1}. This was the approach taken in an earlier version of the present paper (\href{http://arxiv.org/abs/1207.2364}{arXiv:1207.2364v1}), before the appearance of \cite{gbun1,gbun2}.
\end{remark}

\section{The second homology as fundamental group}
\label{sec:h2vsp1}

We now show how homotopy invariance for homology of linear groups can be used to identify the fundamental group of the singular resolution $G(k([\Delta^\bullet]))$ with the second group homology. We define for an isotropic reductive group $G$ simplicial groups $G(k[\Delta^\bullet])$, $\op{E}^G(k[\Delta^\bullet])$ and $\St^G(k[\Delta^\bullet])$ associated to the group, its elementary subgroup and its Steinberg group. The first thing to note is that homotopy invariance of $K_1^G$ implies an isomorphism $\pi_1(G(k[\Delta^\bullet]))\cong\pi_1(\op{E}(k[\Delta^\bullet]))$, which allows us to work with $\op{E}(k[\Delta^\bullet])$ henceforth. We define further simplicial objects: denote by $K_2^G(k[\Delta^\bullet])$ the singular resolution of the functor
\[
A\mapsto K_2^G(A):=\Ker\left(\St^G(A)\rightarrow \op{E}^G(A) \right),
\]
by $\op{UE}^G(k[\Delta^\bullet])$ the singular resolution of the functor $A\mapsto \op{UE}^G(A)$ which assigns to each algebra $A$ the universal central extension $\op{UE}^G(A)$ of the perfect group $E^G(A)$,
and by $\op{H}_2^G(k[\Delta^\bullet])$ the singular resolution of the functor 
\[
A\mapsto \op{H}_2^G(A):=\op{H}_2(G(A),\ZZ)=\Ker\left(\op{UE}^G(A)\rightarrow \op{E}^G(A)\right). 
\]
We chose slightly unusual notation in $\op{H}_2^G$ to distinguish the above object from $\op{H}_2(G(k[\Delta^\bullet]),\ZZ)$ which has a different meaning. 

With these notations, we have the following:

\begin{lemma}
\label{lem:fibseq}
There are fibre sequences of simplicial sets:
\[
\op{H}_2^G(k[\Delta^\bullet])\rightarrow \op{UE}^G(k[\Delta^\bullet])\rightarrow
\op{E}^G(k[\Delta^\bullet]), \textrm{ and}
\]
\[
K_2^G(k[\Delta^\bullet])\rightarrow
\St^G(k[\Delta^\bullet])\rightarrow \op{E}^G(k[\Delta^\bullet]).
\]
\end{lemma}

\begin{proof}
It follows from Moore's lemma, e.g. \cite[Lemma I.3.4]{goerss:jardine}, that the morphisms $\op{UE}^G(k[\Delta^\bullet]) \rightarrow \op{E}^G(k[\Delta^\bullet])$ and $\St^G(k[\Delta^\bullet]) \rightarrow \op{E}^G(k[\Delta^\bullet])$ are fibrations of fibrant simplicial sets. The fibres are by definition $\op{H}_2^G(k[\Delta^\bullet])$ and $K_2^G(k[\Delta^\bullet])$, respectively. 
\end{proof}

\begin{proposition}
\label{prop:pi1h2}
Let $k$ be an infinite field, and let $G$ be an isotropic reductive group over $k$. Then the boundary morphism $\Omega \op{E}^G(k[\Delta^\bullet])\rightarrow \op{H}_2^G(k[\Delta^\bullet])$ associated to the fibration $\op{UE}^G(k[\Delta^\bullet])\rightarrow \op{E}^G(k[\Delta^\bullet])$
induces an isomorphism:
\[
\pi_1(\op{E}^G(k[\Delta^\bullet]),1) \isoto \op{H}_2(G(k),\ZZ). 
\]
If the Steinberg group does not have non-trivial central extensions,
i.e. for all $n$
\[
\St^G(k[\Delta^n])/\left[K_2^G(k[\Delta^n]),\St^G(k[\Delta^n])\right] \rightarrow \op{E}^G(k[\Delta^n])
\]
is the universal central extension, then the boundary morphism $\Omega \op{E}^G(k[\Delta^\bullet])\rightarrow K_2^G(k[\Delta^\bullet])$ associated to the fibration $\St^G(k[\Delta^\bullet])\rightarrow \op{E}^G(k[\Delta^\bullet])$ induces an isomorphism 
\[
\pi_1(\op{E}^G(k[\Delta^\bullet]),1) \isoto K_2^G(k).
\]
\end{proposition}

\begin{proof}
By \cite[Theorem 1.1]{hoinv}, all the usual maps (inclusion of constants, evaluation at $0$) induce isomorphisms $\op{H}_2(G(k),\ZZ)\cong \op{H}_2(G(k[T]),\ZZ)$.
Therefore, we have
\[
\pi_0(\op{H}_2^G(k[\Delta^\bullet]))=\op{H}_2(G(k),\ZZ), \textrm{ and }
\pi_1(\op{H}_2^G(k[\Delta^\bullet]))=0.
\]
Moreover, $\op{E}^G(k)$ and $\St^G(k)$ are generated by $X_\alpha(u)$, $u\in V_\alpha$. These elements are all homotopic to the identity by the homotopy $X_\alpha(uT)$. Therefore, 
\[
\pi_0(\op{E}^G(k[\Delta^\bullet]))\cong \pi_0(\St^G(k[\Delta^\bullet]))=0.
\]
The long exact sequence associated to the fibre sequence from \prettyref{lem:fibseq} yields via the above computations a short exact sequence
\[
0\rightarrow \pi_1(\op{UE}^G(k[\Delta^\bullet]))\rightarrow
\pi_1(\op{E}^G(k[\Delta^\bullet])) \rightarrow
\pi_0(\op{H}_2^G(k[\Delta^\bullet]))\rightarrow 0.
\]
Now let $\widetilde{\op{E}^G}(k[\Delta^\bullet]) \rightarrow \op{E}^G(k[\Delta^\bullet])$
be the universal covering of the simplicial group $\op{E}^G(k[\Delta^\bullet])$.
This has the structure of a simplicial group, and by uniqueness of liftings is degree-wise a central extension by $\pi_1(\op{E}^G(k[\Delta^\bullet]))$. Therefore, the above injective map factors as $\op{UE}^G(k[\Delta^\bullet])\to \widetilde{\op{E}^G}(k[\Delta^\bullet])\to \op{E}^G(k[\Delta^\bullet])$, 
which together with $\pi_1(\widetilde{E^G}(k[\Delta^\bullet]))=0$ 
implies the required isomorphism.

The second claim concerning $K_2$ follows by the same argument,
replacing $\op{UE}^G$ by
\[
\St^G(k[\Delta^n])/[K_2^G(k[\Delta^n]),\St^G(k[\Delta^n])]. \qedhere
\]
\end{proof}

\begin{remark}
It should be noted that the isomorphism in \prettyref{prop:pi1h2} has been established in the case of Chevalley groups over algebraically closed fields in \cite[Theorem 2.1]{jardine:homotopy}.
Jardine's proof uses the spectral sequence for the homology of $G(k[\Delta^\bullet])$ to establish this isomorphism.
This is not too far away from our proof above, however, there are better methods available now to establish the necessary $\Ao$-invariance of $H_2$.
\end{remark}

\section{Explicit description of loops and relations}
\label{sec:explicit}

Fix a root system $\Phi$. For a commutative unital ring $R$ denote $G(\Phi,R)$ the split Chevalley group, $\op{E}(\Phi,R)$ its elementary subgroup and $\op{St}(\Phi,R)$ its Steinberg group. We now describe explicit loops in $\pi_1(G(\Phi,k[\Delta^\bullet]))$, which is a direct translation of the Steinberg symbols for $\op{H}_2$. This also gives rise to an explicit isomorphism
$\op{H}_2(G(\Phi,k),\ZZ)\isoto \pi_1(G(\Phi,k[\Delta^\bullet]),1)$

\begin{definition}
\label{def:exploops}
For every $\alpha \in \Phi$ we denote by $x_\alpha(u)$ the corresponding root group elements and then define morphisms 
\begin{align*}
& X^\alpha \colon \Ga(R) \to \op{E}(\Phi,R[T]),\ & R \ni u        & \mapsto X^\alpha_T(u) := x_\alpha(Tu),\\ 
& W^\alpha \colon \Gm(R) \to \op{E}(\Phi,R[T]),\ & R^\times \ni u & \mapsto W^\alpha_T(u) := X^\alpha_T(u)X^{-\alpha}_T(-u^{-1})X^\alpha_T(u),\\ 
& H^\alpha \colon \Gm(R) \to \op{E}(\Phi,R[T]),\ & R^\times \ni u & \mapsto H^\alpha_T(u) := W^\alpha_T(u)W^\alpha_T(1)^{-1},
\end{align*}
\begin{align*}
& C^\alpha \colon \Gm\times\Gm \to \op{E}(\Phi,R[T]), \\
& \qquad R^\times \times R^\times \ni (a,b) \mapsto C^\alpha_T(a,b) := H^\alpha_T(a)H^\alpha_T(b)H^\alpha_T(ab)^{-1} \in \op{E}(\Phi,R[T]). 
\end{align*}
We will use the same letters with an additional tilde to denote the corresponding lifts to $\St(\Phi,R[\Delta^\bullet])$.
\end{definition}

\begin{example}
We give an example of the ``symbol loops'' in the group $SL_2$.
With the obvious choice $x_\alpha(u)=e_{12}(u)$,
we have
\[
C^\alpha_T(u,v)=\left(\begin{array}{cc}
1&0\\0&1\end{array}\right)+
T(T^2-1)\frac{(1-u)(1-v)}{u^2v}D_T^\alpha(u,v), 
\text{ where}
\]
\[D_T^\alpha(u,v)=
\left(\begin{array}{cc}
u(1-u)T(T^2-1)(T^2-2)&
-vu^2((T^2-1)^2(1-u)+u)(T^2-2)\\
(1-u)(T^2-1)^2-1&-uv(1-u)T(T^2-1)(T^2-2)
\end{array}\right) \qedhere
\]
\end{example}

\begin{remark}
Philosophically, what is happening here is the following: choosing a maximal torus $S$ in $G$, associated root system and root subgroups $x_\alpha$ allows to write down a contraction of the (elementary part of the) torus, i.e. a homotopy $H \colon S\times\Ao\rightarrow G$, where $H(-,0)$ factors through the identity $1\in G$ and $H(-,1)$ is the inclusion of $S$ as maximal torus of $G$. This is nothing but a more elaborate version of the lemma of Whitehead. After fixing such a contraction, there is a preferred choice of path $H(u)$ for any $u\in S$. Given two units in the torus, one can concatenate the paths $H(u)$, $uH(v)$ and $H(uv)^{-1}$ to obtain a loop. This is basically what happens in \prettyref{def:exploops}.
\end{remark}

The translation between elements (and symbols) in the Steinberg group and loops (and symbol loops) in the singular resolution $G(\Phi,k[\Delta^\bullet])$ is given as in covering space theory:

(i) An element of the Steinberg group is given by a product $\tilde{y}=\prod_i \widetilde{x_{\alpha_i}}(u_i)$. Setting $y_T=\prod_i x_\alpha(Tu_i)$ produces a path in $\op{E}(\Phi,R[T])$. If $\tilde{y}$ is in the kernel of the projection $\St(\Phi,R) \rightarrow \op{E}(\Phi,R)$, the path $y_T$ is in fact a loop.

(ii) A path $y_T\in \op{E}(\Phi,k[T])$ with $y_T(0)=1$ can be factored as a product of elementary matrices $\prod_i x_{\alpha_i}(f_i(T))$, which in turn can be lifted to $\St(\Phi,k[T])$. Evaluating at $T=1$ yields an element $\prod_i\widetilde{x_{\alpha_i}}(f_i(1))\in \St(\Phi,R)$. If the path $y_T$ was in fact a loop, then the resulting element $\prod_i\widetilde{x_{\alpha_i}}(f_i(1))\in \St(\Phi,R)$ lies in fact in the kernel of the projection $\St(\Phi,R)\rightarrow \op{E}(\Phi,R)$.

It is then possible to derive elementary relations between the above loops in just the same way as the relations for Steinberg symbols in
\cite{matsumoto}. 
The contraction of the torus $H^\alpha_T(u)$ is chosen such that $H^\alpha_T(1)$ is the constant loop.
From this, it follows immediately that
$C^\alpha_T(x,1)=C^\alpha_T(1,x)=1$
for all $x,y\in k^\times$.
The symbol loops $C^\alpha_T(x,y)$ in $G(\Phi,k[T])$
are not central on the nose,
but are central up to homotopy because the fundamental group of a simplicial group is abelian,
and conjugation by paths acts trivially on the fundamental group.
Then the conjugation formulas in \cite[Lemma 5.2]{matsumoto}
can be translated into statements of homotopies between corresponding products of paths
$W^\alpha_T(u)$ resp. $H^\alpha_T(u)$.
In particular, (weak) bilinearity of symbol loops in the fundamental group can be proved exactly as in \cite{matsumoto}.
For details,
cf. \cite{voelkel:diplom}. 
It is not clear how to prove the Steinberg relation simply by computing with loops and homotopies inside $\op{E}(k[\Delta^\bullet])$. We derive a general Steinberg relation from $\Ao$-homotopy theory in the next section. 

\section{The Steinberg relation from \texorpdfstring{$\Ao$}{A1}-homotopy theory}
\label{sec:steinberg}

In the case of split groups, the Steinberg relation in $\op{H}_2(G(k),\ZZ)$ can be deduced from $\Ao$-homotopy as follows. We denote by $\Sigma$ and $\Omega$ the simplicial suspension and loop space functors, respectively. 

\begin{proposition}
\label{prop:steinberg}
 Let $C \colon \Gm \wedge \Gm \to \Omega G_\bullet$ be any morphism with $G_\bullet$ a simplicial group satisfying affine Nisnevich excision. Let $s \colon \Ao\setminus\{0,1\} \to \Gm \wedge \Gm$ be the Steinberg morphism $a \mapsto (a,1-a)$. Then the composition of $C$ with the Steinberg morphism $C\circ s \colon \Ao\setminus\{0,1\} \to \Omega G_\bullet$ has trivial homotopy class in the simplicial and $\Ao$-local homotopy category.
\end{proposition}

\begin{proof}
We have the natural adjunction $[\Sigma X,Y]\cong[X,\Omega Y]$ both in the simplicial and $\Ao$-local homotopy category. 
Choose a fibrant resolution $r\colon G_\bullet\rightarrow \exinf(G_\bullet)$.
Under the adjunction, the morphism $r\circ C\circ s$ corresponds to the composition
\[
\Sigma \Ao\setminus\{0,1\}\stackrel{\Sigma s}{\longrightarrow}
\Sigma\Gm\wedge\Gm \stackrel{C^{ad}}{\longrightarrow}
\exinf(G_\bullet).
\]
By \cite[Prop. 1]{hu:kriz},
this composition factors through the $\Ao$-contractible space $\Sigma\Ao$ and is therefore trivial.
More specifically,
we have the following equality in
$[\Sigma\Ao\setminus\{0,1\},\exinf(G_\bullet)]_{\Ao}$: 
\[
r\circ C^{ad}\circ \Sigma s=r\circ C^{ad}\circ \Sigma\tilde{s}\circ
\Sigma\iota=r\circ C^{ad}\circ 0=0.
\]
This implies the $\Ao$-local statement. The simplicial statement follows from \prettyref{lem:bgiso} and \cite[Theorem 3.3.5]{gbun1}, which give a bijection
\[
[\Ao\setminus\{0,1\},G_\bullet]_s\cong
[\Ao\setminus\{0,1\},\exinf(G_\bullet)]_{\Ao}.
\qedhere\]
\end{proof}

The result implies that for split $G$ all the loops $C^\alpha(u,1-u)$,
$u\in k^\times$, described in \prettyref{sec:h2vsp1} are contractible
in the singular resolution $G(k[\Delta^\bullet])$: the symbol $C^\alpha(x,y)$ can be interpreted as a morphism of simplicial groups $\Gm\times\Gm\rightarrow\Omega \Sing G$. But since
$C^\alpha(1,y)=C^\alpha(x,1)=1$ is equal to the identity, this
morphism factors through a morphism of simplicial presheaves
$\Gm\wedge\Gm\rightarrow\Omega \Sing G$. The above corollary then
yields the Steinberg relation. Even better, since $\Sing G$ has affine excision, there is a single algebraic morphism $\Ao\setminus\{0,1\}\times\Ao\rightarrow G$ realizing all the Steinberg loops $C^\alpha(u,1-u)$, $u\in k^\times\setminus\{1\}$ at once; and there is a single algebraic homotopy $(\Ao\setminus\{0,1\}\times\Ao)\times\Ao\rightarrow G$ 
providing all the contractions of Steinberg loops at once. 
This is one instance where a computation in group homology can be deduced from $\Ao$-homotopy theory. 

We want to point out the following generalization of the Steinberg relation for non-split groups. Let $D$ be a central simple algebra over $k$. There is an associated reduced norm which can be interpreted as a regular morphism $\operatorname{Nrd}_D\colon\mathbb{A}^{\dim D}\rightarrow\Ao$. In $\mathbb{A}^{\dim D}$ we have two open subschemes, the linear algebraic group $\op{GL}_1(D)$ defined by $\op{Nrd}_D(u)\neq 0$, and another open subscheme $\mathcal{U}_D$ defined by $\op{Nrd}_D(u)\neq 0$ and $\op{Nrd}_D(1-u)\neq 0$. There is an obvious analogue of the Steinberg morphism 
\[
s_D\colon\mathcal{U}_D\rightarrow \op{GL}_1(D)\times \op{GL}_1(D)\rightarrow
\op{GL}_1(D)\wedge \op{GL}_1(D)\colon u\mapsto (u,1-u).
\]

\begin{proposition}
Let $s_D\colon\mathcal{U}_D\rightarrow \op{GL}_1(D)\wedge \op{GL}_1(D)$ be the Steinberg morphism defined above. Then there exists a space $\mathcal{X}_D$ and a commutative diagram 
\begin{center}
\begin{minipage}[c]{10cm}
\begin{tikzpicture}[scale=1.2,arrows=->]
\node (A01) at (0,1) {$\mathcal{U}_D$};
\node (A) at (0,0) {$\mathbb{A}^{\dim D}$};
\node (GmGm) at (3,1) {$\op{GL}_1(D)\wedge \op{GL}_1(D)$};
\node (X) at (3,0) {$\mathcal{X}_D$};
\draw (A01) to 
 node [midway,left] {$\iota$} (A);
\draw (A) to node [midway,above] {$\tilde{s}_D$} (X);
\draw (A01) to node [midway,above] {$s_D$} (GmGm);
\draw (GmGm) to node [midway,right] {$\psi_D$} (X);
\end{tikzpicture}
\end{minipage}
\end{center}
with the suspension $\Sigma \psi_D$ of $\psi_D$ being an $\Ao$-local weak equivalence. 
\end{proposition}

\begin{proof}
The argument is the same as in \cite[Prop. 1]{hu:kriz}, replacing $\Ao$ by $\mathbb{A}^{\dim D}$, $\Gm$ by $\op{GL}_1(D)$, and $\Ao\setminus\{0,1\}$ by $\mathcal{U}_D$. The varieties $V$ and $W$
have to be replaced by $\mathcal{V}_D=[y-1=x\cdot_Dz,y\neq 0]$ and $\mathcal{W}_D=[x-1=y\cdot_Dz,x\neq 0]$. The space $\mathcal{X}_D$ is then the pushout $\mathcal{V}_D\cup_{\op{GL}_1(D)\times \op{GL}_1(D)}\mathcal{W}_D$.
\end{proof}

This provides an $\Ao$-homotopy proof of the Steinberg relation in $\op{H}_2(\op{SL}_n(D),\ZZ)$, $n\geq 3$. All Steinberg relations are given by a single algebraic map $\mathcal{U}_D\times\Ao\rightarrow \op{SL}_n(D)$, and they are all contracted by a single (inexplicit) algebraic homotopy $(\mathcal{U}_D\times\Ao)\times\Ao\rightarrow \op{SL}_n(D)$.

\end{document}